\documentclass[11pt,leqno]{article}

\usepackage[pdftex,bookmarks,colorlinks,breaklinks]{hyperref} 
\hypersetup{linkcolor=blue,citecolor=red}
\usepackage{euscript}

\usepackage{graphicx}
\usepackage{booktabs}
\usepackage{bbm}
\usepackage{wrapfig}
\usepackage{datenumber}
\usepackage{tikz}
\usepackage{hyperref}
\usepackage{ifthen}
\usetikzlibrary{decorations}
\usepgflibrary{decorations.fractals}
\usepgflibrary{shapes.symbols}
\usepgflibrary{shapes.misc}
\usepackage{color}
\usepackage{fancybox}
\usepackage{fancyhdr}
\usepackage{dsfont}
\usepackage{amsmath}
\usepackage{amssymb}
\usepackage{amsthm}
\usepackage{xcolor}
\usepackage{fontenc}
\usepackage{epigraph}
\usepackage{mathrsfs}
\usepackage{boxedminipage}
\usepackage{wasysym}
\usepackage{latexsym}
\usepackage[T1]{fontenc}
\usepackage[latin9]{inputenc}
 
\usepackage{verbatim}
\usepackage{xspace}
\usepackage{mathrsfs}
\usepackage{stmaryrd}
\usepackage{hyperref}

\newcommand{\ZZ}{{\mathbb{Z}}}

\newcommand{\TT}{{\mathbb{T}}}

\newcommand{\RR}{{\mathbb{R}}}
\newcommand{\GP}{{\mathbf{P}}}
\newcommand{\GQ}{{\mathbf{Q}}}

\newcommand{\GN}{{\mathbf{N}}}
\newcommand{\GX}{{\mathbf{X}}}
\newcommand{\QQ}{{\mathbb{Q}}}

\newcommand{\GG}{{\mathbf{G}}}

\newcommand{\T}{\mathsf{T}}
\newcommand{\F}{\mathcal{F}}

\newcommand{\G}{{\mathcal{G}}}

\newcommand{\GL}{\mathbf{GL}}

\newcommand{\Ad}{{\mathrm{Ad}}}

\newcommand{\newterm}[1]{\textit{#1}}

\newcommand{\Sing}{\mathscr{S}}
\newcommand{\Gen}{\mathscr{G}}

\theoremstyle{plain}
\newtheorem{theorem}{Theorem}[section]

\newtheorem{lemma}[theorem]{Lemma}
\newtheorem{proposition}[theorem]{Proposition}

\theoremstyle{definition}
\newtheorem{definition}[theorem]{Definition}

\newtheorem{remark}[theorem]{Remark}

\renewcommand{\O}{\mathcal{O}}
\newcommand{\SK}{\mathscr{K}}
\newcommand{\SF}{\mathscr{F}}
\newcommand{\SC}{\mathscr{C}}

\newcommand{\evit}{\mathscr{E}}
\newcommand{\sing}{\mathscr{S}}
\newcommand{\N}{\mathcal{N}}
\newcommand{\w}{\mathbf{w}}
\newcommand{\m}{\mathbf{m}}
\newcommand{\M}{\mathbf{M}}
\DeclareMathOperator{\Rep}{\mathfrak{R}}

\providecommand{\ov}{\mathbf v}
\providecommand{\oq}{\mathbf q}

\newcommand{\SP}{\mathscr{P}}

\newcommand{\ol}{\mathfrak{O}}

\usepackage[a4paper,top=3cm,bottom=3cm,right=2cm,left=2cm,twoside=false]{geometry}
\title{Distribution of orbits of unipotent groups on $S$-arithmetic homogeneous spaces}

\author{Keivan Mallahi-Karai}
\begin{document}
\maketitle
\begin{abstract}
We will prove an $S$-arithmetic version of a theorem of Dani-Margulis on the convergence of ergodic 
averages of a given bounded continuous function, when the initial point is outside certain compact subsets
of the singular set associated to the unipotent flow. 
\end{abstract}

\section{Introduction}

The goal of this note is to prove a version of a theorem of Dani and Margulis in 
an $S$-arithmetic context. In \cite{DM}, Dani and Margulis proved the following uniform version of Ratner's theorem:

\begin{theorem}\label{th:DM}\cite{DM}
Let $G$ be a connected Lie group, and $\Gamma$ be a lattice in $G$. Let $U=(u_t)$ denote
a one-parameter Ad-unipotent subgroup of $G$. Consider the data consisting of a bounded continuous function $\phi: G/\Gamma \to \RR$, a compact set $\SK \subseteq G/\Gamma$, and $ \epsilon>0$. Then there exists a finite number of proper closes subgroups $H_1, \dots, H_k$
such that $H_i \cap \Gamma$ is a lattice in $H_i$ for all $ 1 \le i \le k$, and compact 
sets $\SC_i \subseteq X(H_i, U)$, $1 \le i \le k$, such that for every compact set 
$\SF \subseteq \SK - \bigcup_{i=1}^{k} \SC_i\Gamma/\Gamma$, for all $x \in \SF$ and $T \gg 0$, 
the following holds:
\[ \left|  \frac{1}{T} \int_{[0,T]} \phi(u_t x) \ dt- \int_{G/\Gamma} \phi d\mu   \right| < \epsilon. \]
\end{theorem}

Here, and in the rest of the article, for a closed subgroups $H$ and $W$ of a Lie group $G$, the set $X(H,W)$ is defined by $X(H,W)=\{ g \in G: g^{-1}Wg  \subseteq H \}$. It is clear that if $H \cap \Gamma$ is
a lattice in $H$, then for any $g \in X(H,W)$ the orbit $Wg\Gamma$ is included in $gH\Gamma$, which
is a closed set carrying a finite $H$-invariant measure. Such points are called 
\newterm{singular points}. The set of singular points will be denoted by $\Sing(W)$. The complement
in $G/\Gamma$ of the set of singular points is called the set of genetic points and is denoted
by $\Gen(W)$. In \cite{DM}, it is shown that if $W$ is connected and generated by Ad-unipotent 
elements, then $\Sing(W)$ is the union of $X(H,W)\Gamma/\Gamma$, where $H$ runs over all
closed connected subgroups of $G$, such that $H \cap \Gamma$ is a lattice in $H$, and 
$\Ad(H \cap \Gamma)$ is Zariski-dense in $\Ad(H)$. 

Perhaps one of the striking features of this theorem is that for the equidistribution up to 
an error of size $ \epsilon$ to be achieved, only a compact subset of a union of finitely many 
singular sets need to be removed. In other words, all but finitely many singular orbits behave
as dense ones, for a given test function $\phi$ and error tolerance $ \epsilon$. 
In \cite{DM}, Dani and Margulis use this theorem to give asymptotically exact lower bounds for the 
number of integer vectors in a given ball satisfying $Q(v)\in (a,b)$, where $Q$ varies over
a compact family of indefinite quadratic forms. In accordance with what was said before, all but finitely many \emph{rational} quadratic forms obey the asymptotic behavior for a given tolerance $ \epsilon$. 

Note that $G$ is an arbitrary (and not necessarily algebraic, or even linear) Lie group, and 
$\Gamma$ does not have to be arithmetic. The version of the theorem proven in this paper 
involves $S$-arithmetic groups, which are sufficient for many applications. It turns out that in this setting, a more restricted class of algebraic subgroups can appear as the orbit closures. This class was introduced in \cite{GT} under the name class $\F$ (see definition \ref{cf}). To state the theorem we also need a substitute for the domain of the unipotent flow (or $S$-adic time). The related definitions are given in Section \ref{pre}.

\begin{theorem}\label{main-DM}
Let $\GG$ be a $k$-algebraic group, $G=\GG(k_S)$, and $\Gamma$ an $S$-arithmetic lattice in $G$, and $\mu$ denote the 
$G$-invariant probability measure on $G/\Gamma$. Let $U=\{ (u_v(t_v))_{v \in S}| t_v \in k_v \}$ be a one-parameter unipotent  
$k_T$-subgroup of $G$, and let $\phi: G/ \Gamma \to \RR$ be a bounded continuous function. Let $\SK$ be a compact subset of
$G/\Gamma$, and let $ \epsilon>0$. Then there exist finitely many proper subgroups $\GP_1, \dots, \GP_k$ of class $\F$,
and compact subsets $C_i \subseteq X(P_i, U)$, where $P_i= \GP_i(k_S)$, $1 \le i \le k$, such that the following holds:
for any compact subset $\SF$ of $\SK - \cup_i C_i\Gamma/\Gamma$ there exists $\T_0$ such that for all $x \in \SF$ and 
$\T \succ \T_0$, we have
\[ \left|   \frac{1}{\lambda_S(I(\T) )} \int_{I(\T)} \phi( u(t)x) d\lambda_{T}(t) - \int_{G/\Gamma} \phi d\mu       \right| \le \epsilon. \]
\end{theorem}

The general line of argument is similar to the one in \cite{DM}. There are a number of places in which technicalities arise that need to be handled differently.

\section{Preliminaries}\label{pre}
In this section we will introduce some notation and recall a number of theorems from 
\cite{GT} that will be used in this paper. 

Let $k$ be a number filed, i.e., a finite extension of $\QQ$, and let $v$ be a valuation of 
$k$, and $| \cdot |_v$ denoted the associated norm. A standing assumption is this paper is that $v$ is normalized, i.e. $v(k^{\ast})=\ZZ$. The completion of $k$ with respect to $v$ is denoted by $k_v$. 
The set of elements of $x \in k_v$ satisfying $|x|_v \le 1$ is called the ring of integers of $k_v$.
Note that $(k_v, +)$ is an abelian locally compact group. The Haar measure on $(k_v,+)$ normalized such that it assigns $1$ to the ring of integers of $k_v$ is denoted by $\lambda_v$. 
 
Let us mention in passing that $k_v$ contains a competition of the $p$-adic field $\QQ_p$, where
$p=p(v)$ is a prime when $v$ is non-archimedean and $p= \infty$ when $v$ is archimedean.  
Let $S$ be a finite set of normalized valuations of $k$ containing the set $S_{ \infty}$ of 
Archimedean ones. We write $S_f = S - S_f$, and $k_T= \oplus_{v \in T} k_v$ for any $T \subseteq S$. We will also denote by $\O_S$ (or simply $\O$) the ring 
of $S$-integers in $k$. Likewise, given a subset $T \subseteq S$, we will fix the Haar measure $\lambda_T= \prod_{v \in T} \lambda_v$ on $K_T$. We also equip $K_T$ with the supremum norm
\[ \|(x_v) \| = \sup_{v \in T} |x_v|_v. \]

Throughout this paper, we will use bold upper scale letters (such $\GG$, $\GP$, etc.) for algebraic groups defined over $k$. The $k_S$ points of these groups are denoted by the corresponding letter case (such as $G$, $P$, etc.). Having fixed a $k$-algebraic group $G$, and a set $S$ of places as above, we denote by $\Gamma$ an $S$-arithmetic subgroup of $G$. This means that $\Gamma$ and $\GG( \O_S)$ are commensurable subgroups of $\GG(k)$. When $G/\Gamma$ is a lattice, we denote 
by $\mu$ the unique $G$-invariant probability measure on $G/\Gamma$.

The $S$-arithmetic analogue of Theorem \ref{th:DM} will naturally involve averaging $\phi(u_tx)$, where $t$ ranges over an $S$-interval. Let us give fix some definitions. For a subset $T \subseteq S$, 
let $\T =(T_v)_{v \in T} \in (\RR^+)^T$, and $a=(a_v)_{v \in T} \in k_T$. The $T$-interval centered at $a$ of radius $\T $ is the subset of $k_T$ defined by 
\[ I(a,\T)=\{ (x_v)_{v \in T} \in k_T: |x_v-a_v|_v \le \T(v), \, \forall v \in T \}. \]
For a fixed $v \in S$ and $r>0$, we define the $k_v$-interval
\[ I_v(r)= \{ x \in k_v: |x| \le r \}. \]
Set $\T=(T_v)_{v \in S}$, where $T_v$ is an integral power of $v$ for $v \in S_f$ and a real number for $v \in S_{ \infty}$. Call $\T$ an $S$-time. The magnitude of an $S$-time is defined by 
$$ |\T| = \prod_{v \in S} T_v$$
Note that $\lambda_T(I(a,\T))=|T|$ for all $a \in k_T$. We will also write $m(\T)=\min_{v \in S} T_v$, where $T_v$ is considered as a real number. The set of all $S$-time vectors is denoted by $\TT_S$. 
For $\T=(T_v), \T'=(T'_v) \in \TT_S$, we write $\T \succ \T'$ if $T_v \ge T'_v$ for all $v \in S$. 
We write $\T_i =(T_{v,i})_{ v\in S} \to \infty$ if $T_{v,i} \to \infty$ for each $v \in S$. 
For $v \in S_f$, assume that $ \varpi \in k_v$ is such that $v(\varpi)=-1$. For an interval $L=I_v(r)$, we will write
$ \hat{L}= I_v( \varpi r)$, so that $\lambda_v( \hat{L})= |\varpi|_v \lambda_v(L)$.

\begin{remark}
Let $v$ be a non-archimedean place. The ultrametric property of the norm implies that if $b \in I_v(a,r)$, we have $I_v(b,r)=I_v(a,r)$. This in particular implies that if $J_1$ and $J_2$ are two intervals with a non-empty intersection, then $J_1 \subseteq J_2$ or $J_2 \subseteq J_1$. 
\end{remark}

 We will need the following easy lemma.

\begin{lemma}\label{}
Let $L_1, \dots, L_r \subseteq k_v$ be disjoint intervals. Then 
\[ \sum_{i=1}^r \lambda_v( L_i) \le \lambda_v( \bigcup_{i=1}^{r} \hat{L_i}). \]
\end{lemma}

Note that $\GG(k_v)$ is naturally embedded in $G= \prod_{v \in S } \GG(k_v)$. By a one-parameter $k_v$-subgroup $U_v= \{u_v(t)\}$ of $G$ we mean a non-trivial $k_v$-rational homomorphism $u_v: k_v \to \GG(k_v)$. Let $T \subseteq S$ and for each $v \in T$, let $U_v=\{u_v(t_v): t_v \in k_v \}$ be a one-parameter unipotent $k_v$-subgroup. Then the 
direct sum $u_T: k_T \to G$ defined by $u_T((t_v)_{v \in T})= ( u_v(t_v) )_{v \in T}$ is called a one-parameter unipotent $k_T$-subgroup of $G$. One of the key properties of the unipotent subgroups
is the non-divergence properties of the unipotent flow that plays an essential role in the 
measure classification results for the actions of these groups. The following $S$-arithmetic version of a quantitative non-divergence theorem will later be needed in this paper:

\begin{theorem}[\cite{GT},Theorem 3.3]\label{recurrence}
Let $\GG$ be a $k$-algebraic group, $G=\GG(k_S)$, and $\Gamma$ an $S$-arithmetic lattice in $G$, and $\mu$ denote the 
$G$-invariant probability measure on $G/\Gamma$. Let $U=\{ u_v(t_v)| t_v \in k_v \}$ be a one-parameter unipotent  
$k_T$-subgroup of $G$. Let $ \epsilon >0$ and $\SK \subseteq G/\Gamma$ be a compact set. Then there exists 
a compact subset $\SK$ such that for any $x \in \SK_1$ and any $T$-interval $I(\T)$ in $k_T$, we have 
\[ \frac{1}{\lambda(I(\T))} \lambda_T \{ t \in I(\T)|u(t)x \not\in \SK_1 \} < \epsilon . \]
\end{theorem}

\begin{definition}\label{cf}
A connected $k$-algebraic subgroup $\GP$ of $\GG$ is a subgroup of class $\F$ relative to $S$ if for each proper normal 
$k$-algebraic subgroup $\GQ$ of $\GP$ there exists $v \in S$ such that $(\GP/\GQ)(k_v)$ contains a non-trivial unipotent element. 
\end{definition}

Let $\Gamma$ be an $S$-arithmetic lattice in $G$. 
If $\GP$ is a subgroup of class $\F$ in $\GG$ then for any subgroup $P'$ of finite index in $\GP(k_S)$, 
we have $P' \cap \Gamma$ is an $S$-arithmetic lattice in $P'$. 
The following theorems have been proven in \cite{GT}.

\begin{proposition}[\cite{GT}, Theorem 4.2] \label{small}
Let $M \subseteq k_v^m$ be Zariski closed. Given a compact set $A \subseteq M$ and $ \epsilon>0$, there exists 
a compact set $B \subseteq M$ containing $A$ such that the following holds: for a compact neighborhood $\Phi$ of
$B$ in $k_v^m$, there exists a neighborhood $\Psi$ of $A$ in $k_v^m$ such that for any one-parameter unipotent subgroup $\{u(t) \}$ in $\GL_m(k_v)$, and any $\w \in k_v^m - W_0$, and any interval $I \subseteq k_v$ containing $0$, we have 
\[ \lambda_v \{ t \in I: u(t) \w  \in \Psi \} \le \epsilon \cdot \lambda_v \{ t \in T: u(t) \w \in \Phi \}. \]
\end{proposition}

We will also need the following theorem, which the $S$-adic analogue of Theorem 2 in \cite{DM}. 

\begin{theorem}\label{limit}
Let $\GG$ be a $k$-algebraic group, $G=\GG(k_S)$, and $\Gamma$ be an $S$-arithmetic lattice in $G$. Let $U^{(i)}=\{ u_v^{(i)}(t_v)| t_v \in k_v \}$ be a sequence of one-parameter unipotent $k_S$-subgroup of $G$, such that 
$ u^{(i)}(t) \to u(t)$ for any $t$ as $ i \to \infty$. Let $x_i \to G/\Gamma$ converge to the 
point $x \in G/\Gamma$, and let $\T_i \to \infty$. For any bounded continuous function 
$\phi: G/\Gamma \to \RR$, we have 
\[ \frac{1}{| \T_i |} \int_{I(\T_i)} \phi( u_t^{(i)}x_i)  \; d\lambda_S(t) \to \int_{G/\Gamma} \phi \; d\mu. \]
\end{theorem}
 
Let us briefly sketch the proof of this theorem. The main ingredient of the proof is the 
quantitative non-divergence theorem, whose $S$-adic analogue, Theorem \ref{recurrence}, 
is proven in \cite{GT}. Arguing by contradiction, one assume that there exists a sequence
$x_i$ of points for which the statement is not true. Using the density of the set of generic points,
one can easily show that $x_i$ could be assumed to be generic for $u_t$. Also using the 
quantitative non-divergence, one can prove that there is no escape of mass to infinite, and then 
one easily shows that the limiting measure is invariant under the action of $u_t$. The measure
classification of Ratner will then finish the proof. For details, we refer the reader to
\cite{DM}.

\section{$S$-adic linearization}
Let $\GP$ be a subgroup of class $\F$ in $\GG$. Using Chevalley's theorem, there exists a $k$-rational representation $\rho: G \to \GL(V_\GP)$ such
that $N_{\GG}(\GP)$ equals the stabilizer of a line in $V$ spanned by a vector $\m \in V(k)$. 
This representation and the vector $\m$ is fixed throughout this paper. Let $\chi$ be the 
$k$-rational character of $N_\GG(\GP)$ defined by $\chi(g)m= g.m$, for $g \in N_\GG(\GP)$. 
We denote $\GN= \{ g \in \GG: g \m = \m \}$ and $N= \GN(k_S)$. We also set 
$\Gamma_N=\Gamma \cap N$ and $\Gamma_P= \Gamma \cap N_\GG(\GP)$. The orbit map $\eta:
 \GG \to \GG \m \subseteq V_\G$ is defined by $\eta(g)=g \m$. $\GG m$ is isomorphic to the quasi-affine variety $\GG/\GN$ and $ \eta$ is a quotient map. Set $\GX= \{ g \in \GG: Ug \subseteq g\GP \}$ and let $A_\GP$ denote the Zariski closure of  $\eta( X(P,U))$. 
 Clearly $\GX$ is an algebraic variety of $\GG$ defined over $k_S$ and $\GX(k_S)=X(P,U)$. 
 It is not hard to show (see \cite{GT}) that 
 \[ \eta^{-1}( A_\GP )= X(P,U). \]

It will be useful to consider the map $\Rep: G/\Gamma \to V_{\GP}$ defined as follows. For each $x \in G/\Gamma$, we define
\[ \Rep(x)= \{ \eta_{\GP}(g): g\in G, x=g\Gamma \}. \]
For $D \subseteq A_{\GP}$ and for $ \gamma \in \Gamma$, we define the $\gamma$-overlaps of $D$ by 
\[ \ol^{\gamma}(D)=\{ g\Gamma: \eta_{\GP}(g) \in D, \eta_{\GP}(g \gamma) \in D \} \subseteq G/\Gamma. \]
Finally, we set 
\[ \ol(D)= \bigcup_{\gamma \in \Gamma - \Gamma_{\GP}} \ol^{\gamma}(D) \subseteq G/\Gamma. \]

Throughout this paper, we will use a number of properties of the overlaps. These are formulated
in the following lemma, whose proof is straightforward:

\begin{lemma}\label{}
For $\gamma \in \Gamma$ and $ \gamma_1 \in \Gamma_{\GP}$, and $D \subseteq A_{\GP}$ we have
\begin{enumerate}
\item  $\ol^{e}(D)= \{ x \in G/\Gamma: \Rep(x) \cap D \neq \emptyset \}$.
\item $\ol^{\gamma}(D)=\ol^{\gamma \gamma_1}(D)$. 
\end{enumerate}
\end{lemma}

In this section, we will use the same notation as above. For each subgroup  $\GP$ of class $\F$ relative to $S$, we will
denote $I_{\GP}= \{ g \in \GG: \rho_{\GP}m_{\GP}=m_{\GP} \}$. 
The proof of the following proposition is exactly the same as the proof of Proposition 7.1 in \cite{DM}. 
\begin{proposition}\label{cover}
Suppose $\GP$ is a subgroup of class $\F$ relative to $S$, and $C \subseteq V_{\GP}$ be compact. Assume also that 
$ \SK \subseteq G/\Gamma$ is compact. Then there exists a compact set $ \widetilde{ C} \subseteq G$ such that 
\[ \pi( \widetilde{ C})= \{ x \in \SK: \Rep(x) \cap C \neq \emptyset \}. \]
\end{proposition}

\begin{proposition}\label{overlap}
Let  $\GP$ be a subgroup of class $\F$ relative to $S$ and $D \subseteq A_{\GP}$ be compact. Let $\SK \subseteq G/\Gamma$ be 
compact. Then the family $$\{ \SK \cap \ol^{\gamma}(D) \}_{\gamma \in \Gamma}$$ contains only finitely many distinct elements. 
Moreover, for each $\gamma \in \Gamma$, there exists a compact set $ \widetilde{ C}_{\gamma} \subseteq \eta_{\GP}^{-1}(D) 
\cap \eta_{\GP}^{-1}(D)\gamma$ such that 
\[ \SK \cap \ol^{\gamma}(D)= \pi( \widetilde{ C}_k). \]
\end{proposition}

\begin{proof}
The argument for finiteness from Proposition 7.2. in \cite{DM} can be carried over verbatim to this case. 
\end{proof}

Let us denote by $\evit$ the class of subsets of $G$ of the form
\[ E= \bigcap_{i=1}^r \eta_{ \GP_i}^{-1}( D_i) \]
where $\GP_i$ are subgroup of class $\F$ and $D_i \subseteq A_{\GP_i}$ are compact. For such a set $E$ (together with the 
given decomposition), we denote $\N(E)$ to be the family of all neighborhoods of the form
\[ \Phi=  \bigcap_{i=1}^r \eta_{ \GP_i}^{-1}( \Theta_i) \]
where $ \Theta_i \supset D_i$ are neighborhoods in $V_{\GP_i}$. We will refer to these neighborhoods as components
of $\Phi$. 

We will now prove a theorem which is a stronger version of the theorem in Tomanov. 
\begin{theorem}
Let $\SK \subseteq G/\Gamma$ be compact and $ \epsilon>0$. Given $E \in \evit$, there exists $E' \in \evit$ such that 
the following holds: given $\Phi \in \N(E')$, there exists a neighborhood $ \Omega \supseteq \pi(E)$ such that for any 
one-parameter unipotent subgroup $\{ u_t \}$ of $G$, and any $g \in G$, and $r_0>0$, one of the following holds:
\begin{enumerate}
\item A component of $\Phi$ contains $\{ u(t)g\gamma: t \in I_v(r) \}$ for some $\gamma \in \Gamma$. 
\item For all $r>r_0$, we have
\[ \frac{1}{\lambda_T(\T )} \lambda_T \{ t \in I_v(r) \setminus I_v(r_0): u(t)g\Gamma \in \Omega \cap \SK \} \le \epsilon. \]
\end{enumerate}
\end{theorem}

\begin{proof} It is clear that we can assume that $E= \eta_{\GP}^{-1}(C)$ and that $E$ is $S(v)$-small. We will now proceed by the induction on $\dim \GP$. The result is clearly valid for $\dim \GP=0$. 
Let us assume that it is known for all $\GP$ with dimension at most $n-1$ and that $C \subseteq A_{\GP}$, with 
$\dim \GP=n$. Applying Proposition \ref{small} to the set $C$ (as a compact
subset of the Zariski closed set of $A_{\GP}$), we obtain a compact subset 
$D$ of $A_{\GP}$ such that for a compact neighborhood $\Phi$ of
$D$ in $A_{\GP}$, there exists a neighborhood $\Psi$ of $C$ in $A_{\GP}$ such that for any one-parameter subgroup $\{u_t \}$ of $\GL(V_{\GP})$ and any $\w \in V_{\GP}- \Phi$, and any interval $I \subseteq k_v$ containing $0$, we have 
\[ \lambda_v \{ t \in I: u_t \w  \in \Psi \} \le \epsilon \cdot \lambda_v \{ t \in T: u_t \w \in \Phi \}. \]
Note that since the set of the roots of unity in $K$ is finite, we can choose 
$D$ such that $ \omega D= D$ for every root of unity $ \omega \in K$. Note that $D$ can be chosen to be $S(v)$-small. Now, let $B= \eta_{\GP}^{-1}(D)$. 

By Proposition \ref{overlap} the family
of sets $\{ \SK \cap O^{\gamma}(D) \}_{\gamma \in \Gamma}$ is finite, hence
consisting of the sets $ \SK \cap O^{\gamma_j}(D)$, for $ 1 \le j \le k$. We 
assume that $\gamma_1=e$. Moreover, we can write $  \SK \cap O^{\gamma_j}(D)= \pi(C_j)$ for some
compact subset $C_j \subseteq B \cap B\gamma_j^{-1} \subseteq
X( \GP \cap \gamma_j \GP \gamma_j^{-1}, W)$. We claim that 
$\gamma_j \not\in \Gamma_{\GP}$ for $j \ge 2$. Assuming the contrary, 
we obtain $\rho(\gamma_j)\m_{\GP}=\chi(\gamma_j)\m_{\GP}$. Since
$\chi(\gamma_j) \in \O^{\ast}$, we have
$ \eta(b\gamma_j)= \chi(\gamma_j) \eta(b) \in D$. Since $D$ is $S(v)$-small, we obtain that $\chi(\gamma_j)$ is a root of unity in $k_v^{\ast}$. 
This shows that $B\gamma_j \subseteq \eta^{-1}(D)=B$, which is a contradiction to the choice of $\gamma_j$. This shows that for $j \ge 2$, $\GP 
\cap \gamma_j \GP \gamma_j^{-1}$ is a proper subgroup of $\GP$. Hence there exists a subgroup $\GP_j$ of class $F$ which is contained in the connected component of  $\GP \cap \gamma_j \GP \gamma_j^{-1}$. Note that $\GP_j$ is of dimension less than $n$, and $C_j \subseteq X(\GP_j, W)$. We now set 
$E_j= \eta_{\GP_j}^{-1}(\eta_{\GP_j}(C_j))$ and apply the induction hypothesis to obtain $E_j' \in \evit$ such that for any choice of $\Phi_j 
\in \N(E_j')$, we can find neighborhoods $ \Omega_j$ of $E_j$ such that 
for any one-parameter subgroup $(u(t))_{t \in k_v}$ of $G$, $g \in G$ and
$r>0$, we have 
$$\lambda_{v} \{ t \in I_v(r): u(t)g\Gamma \in \Omega_j 
\cap \SK \} \le (\epsilon/2k)r$$
unless there exists $\gamma \in \Gamma$ such that $\{ u(t)g\gamma: 
t \in I_v(r)$ is contained in a component of $\Phi_j$. Set, 
$E''= \bigcup_{j=2}^n E'_j \in \evit$, and $E'=E'' \cap B$. Consider 
$\Phi \in \N(E')$. This shows that the there exists a neighborbood 
$ \Omega'$ of $\pi(E'')$ such that for any one-parameter unipotent subgroup
$\{ u(t) \}_{t \in k_v}$, and every $g \in G$ and $r_0>0$, we have
\[ \lambda_v( t \in I_v(r_0): t(t)g\Gamma \in \Omega' \cap \SK \}
\le \frac{\epsilon r}{2}\]
unless $\{ u(t)g\gamma: t \in I_v(r_0) \}$ is in a component of $\Phi$ for some $\gamma \in \Gamma$. 

Set $\SK_1 = \SK - \Omega'$, and choose a compact subset $K' \subseteq G$ 
such that $\pi(K')=\SK_1$.  Let $\Phi_1$ be a neighborhood of $D$ in $V$ 
such that $ \eta_{\GP}^{-1}(\Phi_1) \subseteq \Phi$ and
$O( \Phi_1) \cap \SK_1 =\emptyset$. Since $D$ is $S(v)$-small, we can clearly choose $\Phi_1$ to be $S(v)$-small. Note that since $  \rho(u(t))$ is
a one-parameter unipotent subgroup of $\GL(V)$, and $ \eta_{\GP}(C)$
is of relative size less than $ \epsilon/4$ in $D$, we can find a 
neighborhood $\Psi$ of $C$ in $V$ satisfying
\[ \lambda_v( \{ t \in I_v(r): \rho( u(t) )v \in \Psi \}) 
\le \frac{\epsilon}{4} \lambda_v( \{ t \in I_v(r): \rho( u(t))\ov \in \Phi_1 \}),\]
for all $\ov \in V- \Phi_1$, $r>0$ and unipotent subgroups $\{ u(t) \}_{t \in k_v}$. Let $ \Omega= \pi( \eta_{\GP}( \Psi) ) \subseteq G/\Gamma$. Assuming that (1) does not hold for $g \in G$, a one-parameter subgroup $\{ u(t) \}_{t \in k_v }$, and $r_0>0$. This implies that for every $\gamma \in \Gamma$, there exists $t \in I_v(r_0)$ such that $u(t)g\gamma \in G- \Phi$. 
For $q \in \mathbf{M}$, we consider the following sets:
\[ J_1(\oq)= \{ t \in I_v(r)- I_v(r_0): \rho( u(t)g)\oq \in \Phi_1 \}. \]
\[ J_2(\oq)= \{ t \in I_v(r)- I_v(r_0): \rho( u(t)g)\oq \in \Psi, \pi(u(t)g) \in
\SK_1 \}. \]
Note that $J_1(\oq)$ is an open subset of $k_v$ and is hence a disjoint union of intervals. 
We will also define $J_3(\oq) \subseteq J_1(\oq)$ as follows: if $v$ is an archimedean place, then $J_3(\oq)$ consists of those $ t \in J_1(\oq)$ such that 
for some $a \ge 0$, we have $[t,t+a] \subseteq J_1(\oq)$ and $\pi( u(t+a)g) 
\in \SK_1$. If $v$ is a non-archimedean place, then $J_3(\oq)$ consists of those $ t \in J_1(\oq)$ such that there exists an interval $J \subseteq k_v$ containing $t$ and $t' \in J$ such that $\pi(u(t')g) \in \SK_1$. Clearly $J_3(\oq)$ is open
in $k_v$, and is hence a disjoint union of intervals. 
We first make the following claim:

{\bf Claim}: $J_3(\oq_1) \cap J_3(\oq_2) = \emptyset$ for $\oq_1, \oq_2 \in \M$, unless
$\oq_2= \omega \oq_1$ for some root of unity $ \omega \in K^{\ast}$. 

In the archimedean case, if $t \in J_3(\oq_1) \cap J_3(\oq_2)$, then there exists
$a \ge 0$ such that $[t, t+ a] \subseteq J_1(\oq_1) \cap J_1(\oq_2)$ and 
$\pi(u(t+a)g) \in \SK_1$. If $q_j= \eta(\gamma_j)$ for $j=1,2$, then
$ \eta( u(t+a)g\gamma_1)= \eta( u(t+a) g \gamma_2) \in \Phi_1$, we will 
have $\oq_1, \oq_2 \in O( \Phi) \cap \SK_1 =\emptyset$, unless $ \gamma_1^{-1}\gamma_2 \in \Gamma_{\GP}$, which implies that $\oq_2= \omega \oq_1$.

In the non-archimedean case, if $t \in J_3(\oq_1) \cap J_3(\oq_2)$, then $ t \in J_1(\oq)$ and there exist intervals 
$J(\oq_1), J(\oq_2) \subseteq k_v$ containing $t$ and $t_1' \in J(\oq_1)$ and $t_2' \in J(\oq_2)$ 
such that $\pi(u(t'_1)g), \pi( u(t'_2)g) \in \SK_1$. Note that since $J(\oq_1)$ and $J(\oq_2)$ intersect one contains the other, hence,
without loss of generality, we can assume that $t'_1 \in J(\oq_1) \cap J(\oq_2)$, and $\pi(u(t'_1)g) \in \SK_1$. The rest 
of the argument is as in the archimedean case. 

Let $\mathcal{L}_1$ be the family of those components $L=I_v(a,r_1)$ of $J_1(\oq)$ such that $L \cap I_v(r_0) =\emptyset$, 
and $\mathcal{L}_2$ the rest of components.  Note that $ \hat{L} \not\subseteq I_v(r_0)$. This implies 
that 
\[ \lambda_v( \hat{L} \cap J_2(\oq) ) \le \lambda_v( \hat{L} \cap J_3(\oq) ). \]
From here and using the above claim we have
\[ \sum_{L \in \mathcal{L}_1 } \lambda_v( \hat{L} \cap J_2(\oq) ) \le 
\sum_{   L \in \mathcal{L}_1} \lambda_v( \hat{L} \cap J_3(\oq) )  \le  \lambda \lambda_v ( I_v(r) - I_v(r_0) ).   \]

We now claim that 
\[ \sum_{L \in \mathcal{L}_2} \lambda_v(L) \le \lambda_v( I_v(r)). \]
In fact, if $L \in \mathcal{L}_2$, then either $L \subseteq I_v(r_0)$ or $I_v(r_0) \subseteq L$. 
If $I_v(r_0) \subseteq L$ for some $L \in \mathcal{L}_2$, then since components are disjoint, $\mathcal{L}_2$ has precisely one element and the result follows. So, assume that for each $L \in  \mathcal{L}_2$, we have
$L \subseteq I_v(r_0)$. Then the disjointness of components imply that 
\[  \sum_{L \in \mathcal{L}_2} \lambda_v(L) \le \lambda_v( I_v(r_0)). \]
\end{proof}


\begin{proposition}\label{}
Let $\GG$ be a $k$-algebraic group, $G=\GG(k_S)$, and $\Gamma$ an $S$-arithmetic lattice in $G$.
Let $U$ be a one-parameter unipotent subgroup of $G$. Assume that 
$\GP_1, \dots, \GP_k$ are subgroups of class $\F$ for $ 1 \le i \le k$, and let 
$D_i$ be a compact subset of $A_{\GP_i}$, and $\Theta_i$ be a compact neighborhood of
$D_i$ in $V_{\GP_i}$. For a given compact set $\SK \subseteq G/\Gamma$, there exists 
$\GP'_1, \dots, \GP'_k$ of class $\F$ and compact subsets 
$D'_i$ be a compact subset of $A_{\GP'_i}$, $ 1 \le  i \le k$, such that for any 
compact set $ \SF \subseteq \SK - \bigcup_{i=1}^{k} ( \eta_{\GP_i}^{-1}( D_i) \cup
\eta_{\GP'_i}^{-1}(D'_i) )\Gamma/\Gamma$, there exists $\T_0$ such that 
for any $g \in G$ with $g\Gamma \in \SF$, and $ 1 \le i \le k$, there exists $t\in B(\T_0)$ such that 
$u_tg \not\in \eta_{ \SP_i}^{-1}( \Theta_i)$. 
\end{proposition}

\begin{proof}
The proof of this proposition is very similar to the proof of Proposition 8.1. in \cite{DM}. 
Let us denote by $I_\GP$ consists of $g \in G$ with $g.\m_{\GP}=\m_{\GP}$. We first claim that there exists a subgroup $\GP'$ of class $\F$ such that $X(I_{\GP}, U) \subseteq X(\GP',U)$. In fact, let 
$\GP'$ be the smallest connected algebraic subgroup of $\GG$ which contains all the unipotent elements of $I_{\GP}$. Note that since $\GP'$ is generated by unipotent subgroups, 
we have $X_k(\GG)=\{ 1 \}$, where $X_k(\GG)$ denotes the group of characters of $G$ defined over $k$. It follows from Theorem 12.3. of \cite{BH} that $\GG' \cap \Gamma$ is a lattice in $\GG'$. 
We can now show that $\GG'$ is of class $\F$. 
Let $\GP'_1, \dots, \GP'_k$ be chosen as above such that $X(I_{\GP_i}, U) \subseteq X(\GP', U)$
for all $1 \le i \le k$. We will also define 
\[ Q_i= \{ w \in \Theta_i: \rho_{\GP_i}(u_t) w=w, \ \forall t \}. \]
Using Proposition \ref{cover}, we can find compact subsets $C_i \subseteq \epsilon_{\GP_i}^{-1}(Q_i)$, $1 \le i \le k$ such that 
\[ \pi( \widetilde{ C_i})= \{ x \in Q_i: \Rep(x) \cap C_i \neq \emptyset \}. \]
This implies that $C_i \subseteq X( \GP_i, U)$. Consider the compact sets $D_i'= \eta_{\GP'_i}(C_i) \subseteq A_{\GP'_i}$ and assume that $\SF$ is a compact subset of $\SK - 
\pi( \bigcup_{i=1}^{k} ( \eta_{\GP_i}^{-1}(D_i) \cup \eta_{\GP_i'}^{-1}(D'_i))$ is given. Find a compact subset $F' \subseteq G$ such that $\SF= \pi(F')$. From the fact that 
$ \rho_{\GP_i}(\Gamma) \m_{\GP_i}$ is a discrete subset of $V_{\GP_i}$, it follows that there
are only finitely many $ \gamma \in \Gamma$ such that $\eta_{\GP_i}( \gamma) \in \rho_{\GP_i}(F')^{-1}\Theta_i$ for some $1 \le i \le k$. It thus suffices to show that for each 
$1 \le i \le k$ and $\gamma \in \Gamma$, for all large enough $\T$ the set
$ \Theta_i \cap \rho_{\GP_i}(u_i)( \Theta_i \cap \rho_{\GP_i}(F'\gamma) \m_{\GP_i}) =\emptyset$.
Note that $ \rho_{\GP_i}(F'\gamma) \m_{\GP_i} \cap \Theta_i$ is a compact subset of $ V_{\GP_i}$
which does not contain any fixed point of the flow $ \rho_{\GP_i}(u_t)$. Since $ \rho_{\GP_i}(u_t)$ is a unipotent one-parameter subgroup of $\GL(V_{\GP_i})$ the result follows.

\end{proof}

\begin{proof}[Proof of Theorem \ref{main-DM}]

%
%

For a bounded continuous function $\phi$ defined on $G/\Gamma$, the one-parameter unipotent group $(u_t)$ and 
time box $\T $, and $x \in G/\Gamma$, we define
\[ \Delta(\phi,u_t,x,\T )=  \left|   \frac{1}{\lambda_T(I(\T) )} \int_{I(\T)} \phi( u_tx) d\lambda_{T}(t) - \int_{G/\Gamma} \phi d\mu       \right|. \] 
Let us consider the above statement with $S$ replaced by $T \subseteq S$ everywhere. We will prove the statement first for the case $|T|=1$. Then we will show that if 
the statement holds for $T_1$ and $T_2$, then it must also hold for $T_1 \cup T_2$. Let us start with the case $T= \{ v \}$ for some $v \in S$. 
We argue by contradiction. Assume that the statement of the theorem is not true. This implies the existence of a
bounded continuous function $\phi: G/\Gamma \to \RR$, a compact subset $\SK_1 \subseteq G/\Gamma$, and $ \epsilon>0$ such that 
for any proper subgroups $\GP_1, \dots, \GP_k$ of class $\F$, and any compact subsets $C_i \subseteq X(P_i, U)$, where $P_i= \GP(k_S)$, $1 \le i \le k$, there exists a compact set  $\SF$ of $\SK - \cup_i C_i\Gamma/\Gamma$  such that for all $T_0>0$, 
there exists $\T $ with $m(\T )>T_0$, and $x \in F$ such that 
\[ \Delta(\phi,u_t,x,\T ) > \epsilon. \]


Without loss of generality, we can assume that $\phi$ has a compact support, and $\|\phi \| \le 1$. There exists a compact subset $\SK \subseteq G/\Gamma$
such that for all $x \in \SK_1$ and $\T $, we have 
\begin{equation}\label{time-in-K}
 \lambda_T \{ t \in I(\T): u_tx \not\in \SK \} < \frac{1}{3} \lambda_T(I(\T)).  
\end{equation}
We can now apply to construct an increasing sequence $E_i \subseteq E_{i+1}$ in $\evit$ such that 
\begin{enumerate}
\item The family $\{ E_i \}_{i \ge 1}$ exhausts the singular set of $U$, i.e., $ \bigcup_{i=1}^{ \infty} E_i=
\sing(U)$. 
\item For each $i \ge 1$, there exists an open neighborhood $ \Omega_i \supset E_i\Gamma/E_i$ such that for any compact 
set $F \subseteq \SK- E_{i+1}\Gamma/\Gamma$, there exists $\T _{i+1}$ such that for all $x \in F$ and  
$\T \succ \T '_i$ we have 
\begin{equation}\label{time-in-sing}
 \lambda_T \{ t \in I(\T): u_tx \in \Omega_i \cap \SK \} \le \frac{1}{4^i} \lambda_T(I(\T)). 
\end{equation}

\end{enumerate}

For $i \ge 1$, write $\SK \cap \pi(E_i)= \bigcup \pi(C_{j})$ for some compact sets $C_j \in X(P_j, U)$, $1 \le j \le k$. 
As we are arguing by contradiction, we can find a compact subset $\SF_i \subseteq \SK_1 - \pi(E_{i+1})$ such that for 
each $\T _0$ there exists $x \in \SF_i$ and $\T  \ge \T _0$ such that $ \Delta(\phi,u_t,x,\T ) > \epsilon$. 
Without loss of generality, assume that $|\T _1| \le |\T _2| \le \cdots $. This implies that there exists $x_i \in F_i$ and $\sigma_i $ such that $\Delta(\phi, u_t,x_i, \sigma_i)> \epsilon. $
From \eqref{time-in-K} and \eqref{time-in-sing}, we obtain for each $j \ge 1$, a time $t_j \in I(\T _j)$ such that $u_{t_{j} }  x_j \in \SK - \bigcup_{i=1}^{j} \Omega_i$. This implies that 
\[ \Delta( \phi, u_t, \sigma_j, y_j) \ge \epsilon  - 2 \frac{|\T _j|}{| \sigma_j|}   \ge \frac{\epsilon}{3}. \] 
As $y_j \in \SK$ and $\SK$ is compact, there exists a limit point $y \in \SK$. By construction, $y \not\in \Omega_j$ for all $j \ge 1$. This shows that $y$ is not a singular point for $U$. Now, we can apply Theorem \ref{limit} to the convergent subsequence of 
$\{ y_j \}$ and the corresponding subsequence of $ \sigma_i$, to obtain a contradiction. 
 Let us now turn to the general case. 
Assume that the statement is known for $T_1, T_2 \subseteq S$, and $T_1 \cap T_2 =\emptyset$. We write $U_1=( u_v(t_v))_{ v \in T_1}$ 
and $U_2=( u_v(t_v) )_{v \in T_2}$. 
Note that there exists a compact subset $\SK_1 \subseteq G/\Gamma$ such that for all
$x \in \SK$ and any interval $I(\T) \subseteq K_{T_j}$, $j=1,2$, we have 
\[ \frac{1}{\lambda_{j}(I(\T))} \lambda_{j} \left\{  t \in I(\T): u_j(t) \in \SK_1  \right\} \ge 1 - \epsilon/16. \]
Here, we have used the shorthands $u_1(t)=(u_v(t_v))_{v \in T_1}$ and $d\lambda_1$ for the Haar measure on $ \prod_{v \in T_1} k_v$.

By the induction hypothesis, there exist finitely many proper subgroups $\GP_1, \dots, \GP_k$ of class $\F$,
and compact subsets $C_i \subseteq X(P_i, U_1)$, where $P_i= \GP_i(k_S)$, $1 \le i \le k$, such that the following holds:
for any compact subset $\SF$ of $\SK_1 - \cup_i C_i\Gamma/\Gamma$ there exists $T_0$ such that for all $x \in \SF$ and 
$\T $ with $m(\T )> T_0$, we have
\[ \left|   \frac{1}{\lambda_1(I(\T) )} \int_{I(\T)} \phi( u_1(t)x) d\lambda_1(t) - \int_{G/\Gamma} \phi d\mu       \right| \le \frac{\epsilon}{16}. \]

Since $C_i\Gamma/\Gamma \subseteq G/\Gamma$ has measure zero, we can choose neighborhoods $N_i$ of $C_i\Gamma/\Gamma$ of measure at most $ \epsilon/16k$. Now, let $\phi_i$, $1 \le i 
\le k$ be a continuous function such whose restriction to $U_i$ is $1$, and 
$\int_{G/\Gamma} \phi_i < \epsilon/8k$. By applying the induction hypothesis to $ \phi_1, \dots, \phi_k$, we can find 
 finitely many proper subgroups $\GQ_1, \dots, \GQ_l$ of class $\F$,
and compact subsets $D_i \subseteq X(Q_i, U_2)$, where $Q_i= \GQ_i(k_S)$, $1 \le i \le l$, such that the following holds:
for any compact subset $\SF$ of $\SK - \cup_i D_i\Gamma/\Gamma$ there exists $T_1$ such that for all $x \in \SF$ and 
$\T _2$ with $m(\T _2)> T_1$, we have
\[ \left|   \frac{1}{\lambda_2(I(\T) )} \int_{I(\T)} \phi_i( u_2(t)x) d\lambda_2(t) - \int_{G/\Gamma} \phi_i d\mu       \right| \le \frac{\epsilon}{16k}, \quad 1 \le i \le k. \]
Since $\phi_i(x)=1$ for all $x \in U_i$, we obtain
\[ \lambda_2 \left\{ t_2 \in I(\T _2): u_2(t_2)x \in \bigcup_{i=1}^{k} N_i \right\} \le \frac{ \epsilon}{16}. \]

Let $A= \{ t_2 \in I(\T _2): u_2(t_2)x \in \SK_1 \}.$ Note that by the choice of $\SK_1$, 
we have $$\lambda_2(A) \ge (1- \epsilon/16) \lambda_2(I(\T _2)).$$

Combining the last two equations, we obtain 
\[ \lambda_2 \left\{ t_2 \in I(\T _2): u_2(t_2)x \in \SK_1 - \bigcup_{i=1}^{k} N_i \right\}
\ge 1- \frac{\epsilon}{8}. \]

Since  $ \SK_1 - \bigcup_{i=1}^{k} N_i$ is a compact subset of $\SK_1$, disjoint from 
$ \bigcup_{i=1}^{k} C_i\Gamma/\Gamma$, we have 
there exists $T_2$ such that for all $x \in   \SK_1 - \bigcup_{i=1}^{k} N_i$ and 
$\T $ with $m(\T )> T_2$, we have
\[ \left|   \frac{1}{\lambda_1(I(\T _1) )} \int_{I(\T _1)} \phi( u_1(t)x) d\lambda_1(t) - \int_{G/\Gamma} \phi d\mu       \right| \le \frac{\epsilon}{16}. \]

Let us now consider
\[   \frac{1}{\lambda(I(\T) )} \int_{I(\T)} \phi( u(t)x) d\lambda(t) 
= \frac{1}{\lambda_1(I(\T _1) )}  \frac{1}{\lambda_2(I(\T _2) )} \int_{I(\T _1)\times I(\T _2)} \phi( u_1(t_1)u_2(t_2)x) d\lambda_2(t_1) d\lambda_1(t_2). \]

Combining the last two inequalities show that 
\[ \left|   \frac{1}{\lambda(I(\T) )} \int_{I(\T)} \phi( u(t)x) d\lambda(t)- 
\int_{G/\Gamma} \phi   \right| \le \frac{\epsilon}{4}. \]

It follows that the union $X(\GP_i), U), X(\GQ_j,U)$ will satisfy the conditions of the theorem. 
\end{proof}

\end{document}